\documentclass[11pt]{article}
\usepackage[a4paper,margin=1in]{geometry}
\usepackage{amsmath,amssymb,amsthm,mathtools}
\usepackage{graphicx}
\usepackage{booktabs}
\usepackage{microtype}
\usepackage{enumitem}
\usepackage{hyperref}
\usepackage{caption}
\usepackage{subcaption}
\usepackage{xcolor}
\usepackage{bm}
\hypersetup{colorlinks=true,citecolor=blue,linkcolor=blue,urlcolor=blue}
\newtheorem{theorem}{Theorem}[section]
\newtheorem{proposition}[theorem]{Proposition}
\newtheorem{corollary}[theorem]{Corollary}

\theoremstyle{remark}
\newtheorem{remark}[theorem]{Remark}
\newtheorem{definition}[theorem]{Definition}
\newcommand{\dt}{\Delta t}
\newcommand{\dx}{\Delta x}

\newcommand{\norm}[1]{\left\lVert #1\right\rVert}

\title{PDE Realization and Structure-Aware Solvers for a Compact\\Two-Stage Fourth-Order IMEX Method}
\author{Zhixin Huo\\
School of Mathematics and Information Science, Henan Polytechnic University\\
Jiaozuo 454003, Henan, China\\
\texttt{zhixinhuo@hpu.edu.cn}}
\date{}

\begin{document}
\maketitle

\begin{abstract}
This paper develops a PDE realization and solver framework for a compact two-stage fourth-order two-derivative implicit--explicit time discretization for stiff split problems. The key difficulty is a consistency--cost coupling: full-field temporal differentiation is required to preserve the mixed explicit--implicit interactions of the time integrator, but the same interactions widen the implicit stage operators and can make each solve substantially more expensive. We show that a second-order ADER/Cauchy--Kowalevski local evolution is sufficient for the fourth-order outer composition, establish smooth reconstruction consistency and a fully discrete error balance, and derive the leading Lie-bracket defect produced by self differentiation of the split fields. An inexact-stage analysis gives asymmetric midpoint and endpoint residual tolerances that preserve fourth-order accuracy. For the widened stages, the complete mixed action is retained matrix-free while only dominant stiff physics is approximated in the inverse. This yields quadratic and shifted preconditioners, Fourier and multilevel realizations, a semilinear reaction--diffusion reduction, and source-local elimination for relaxation systems. Exact quadratic cancellation, a diffusion-dominated Fourier estimate, and an $\varepsilon/\Delta x$ bound for Jin--Xin relaxation explain the main solver mechanisms. Numerical ablations verify the consistency and tolerance results, while a two-dimensional Brusselator study on two grids shows favorable error-versus-wall-time behavior over a useful accuracy range. A classical stiff-front benchmark also identifies the separate spatial shock--source limitation. The results support a regime-dependent efficiency claim rather than a universal speedup.
\end{abstract}



\section{Introduction}
Compact high-order time integration is attractive for stiff PDEs because the cost of a practical implicit method is determined not only by the formal stage count, but by the algebraic systems created at those stages.  Classical fourth-order IMEX Runge--Kutta methods provide a mature framework for additive splittings~\cite{ascher1997,kennedy2003,pareschi2005,sandu2015}, while multiderivative and Lax--Wendroff-type methods use temporal derivative information to obtain high order with fewer stages~\cite{lax1960,qiu2003,seal2014,christlieb2016}.  Two-stage fourth-order constructions based on local evolution information are by now well established in GRP, gas-kinetic, and related flow solvers~\cite{benartzi2006,li2016,pan2016,du2018,yuan2018}.  ADER methods systematize the same CK philosophy and have been extended to stiff balance laws, WENO discretizations, and locally implicit advection--diffusion--reaction formulations~\cite{toro2002,dumbser2008fv,dumbser2008framework,balsara2009,montecinos2014,toro2015}.

The present paper is a companion PDE-realization study of a compact two-stage fourth-order two-derivative IMEX time discretization developed in a separate manuscript by the present author that is currently under review. The related manuscript establishes the underlying time-stepping formula and its time-level accuracy and stability properties; the present study begins after spatial discretization.  For
\begin{equation}
U_t=F(U)+G(U)=H(U),
\label{eq:split}
\end{equation}
the companion paper establishes the time-stepping formula, fourth-order accuracy when the temporal derivatives are evaluated along the full field,
\begin{equation}
\dot F=F_U(F+G),\qquad \dot G=G_U(F+G),
\label{eq:fullfield}
\end{equation}
and favorable stiff-mode decay in the appropriate limiting regimes.  The question addressed here is different: how should \eqref{eq:fullfield} be realized after spatial discretization, and how can the resulting implicit systems be solved without losing the computational advantage of the compact stage count?

This distinction is important because the same mixed terms that are mathematically required in \eqref{eq:fullfield} also alter the linear algebra.  For linear semi-discrete operators \(F_h=A_hU\) and \(G_h=B_hU\), an implicit stage contains
\[
B_h(A_h+B_h),
\]
which is wider and more expensive than the original stiff operator \(B_h\).  If that product is assembled and factorized naively, ``two implicit solves per step'' need not mean a cheap method.  Our organizing principle is therefore
\begin{center}
\emph{retain the full mixed discretization exactly, but approximate only the dominant stiff physics in the inverse.}
\end{center}
This separates mathematical consistency from linear-algebra approximation.

To keep the boundary between the two companion studies explicit, we treat four ingredients as inherited from the companion time-discretization manuscript: the two-stage fourth-order formula itself, full-field mixed compatibility at the time-discretization level, the fourth-order temporal consistency of the outer method, and the pure-implicit stability function with its quadratic stiff decay.  The present paper begins at the PDE/semi-discrete level.  Its new content is the ADER/CK realization and reconstruction analysis; the explicit Lie-bracket defect of an incorrect self-differentiated realization; the fully discrete and inexact-stage accuracy budgets; the widened mixed-stage operator and its matrix-free, quadratic, FFT, multilevel, semilinear, and source-local solver realizations; the \(\varepsilon/\dx\) relaxation estimate; and PDE-level benchmark and wall-clock evidence.  This division of scope is maintained throughout the paper so that results already established in the companion time-discretization manuscript are used only as inputs or motivation.

The central scientific point is a \emph{consistency--cost coupling}: the cross interaction required to realize the inherited fourth-order method correctly at PDE level is the same interaction that widens the implicit stage and can erase the nominal advantage of having only two stages.  The contributions are therefore threefold.  First, we establish the PDE realization side of this coupling through mixed-compatible ADER2/CK differentiation, smooth reconstruction consistency, the leading commutator defect, and fully discrete/inexact-stage accuracy budgets.  Second, we resolve the resulting algebraic cost through the principle of an exact mixed residual with a structure-aware approximate inverse, supported by exact quadratic cancellation, a diffusion-dominated Fourier bound, semilinear contraction, and the Jin--Xin \(\varepsilon/\dx\) estimate.  Third, we test these mechanisms on standard stiff PDE benchmarks, direct wall-time/error sweeps, and reproducible solver-mechanism experiments, explicitly distinguishing accuracy per expensive solve from wall-clock time-to-accuracy.

The remainder of the paper is organized as follows. Section~2 recalls only the inherited compact time-discretization formulas needed to define the PDE realization target and explains why an order-two local evolution is sufficient. Section~3 constructs the PDE-level CK derivatives and proves their smooth reconstruction consistency. Section~4 derives the leading commutator defect produced by self differentiation. Section~5 develops the fully discrete error balance and the residual tolerances required for inexact stage solves. Section~6 identifies the widened mixed-stage operator that links consistency to computational cost. Section~7 develops the structure-aware solver hierarchy and its algebraic, Fourier, contraction, and relaxation estimates. Section~8 presents the mechanism tests, standard PDE benchmarks, wall-time/error comparisons, and reproducible checks of the solver estimates. Section~9 discusses the scope and limitations of the results, and Section~10 summarizes the conclusions.

\section{Inherited compact time discretization and PDE realization target}
Let \(h=\dt\).  The companion method has midpoint stage
\begin{equation}
U^*=U^n+\frac h2F(U^n)+\frac{h^2}{8}\dot F(U^n)
+\frac h2G(U^*)-\frac{h^2}{8}\dot G(U^*),
\label{eq:mid}
\end{equation}
and endpoint stage
\begin{equation}
\begin{aligned}
U^{n+1}={}&U^n+hF(U^n)+\frac{h^2}{6}\left[\dot F(U^n)+2\dot F(U^*)\right]\\
&+hG(U^{n+1})-\frac{h^2}{6}\left[\dot G(U^{n+1})+2\dot G(U^*)\right].
\end{aligned}
\label{eq:end}
\end{equation}
The time-discretization analysis, mixed compatibility of \eqref{eq:fullfield}, and the stability function are developed in the companion time-discretization manuscript. Here we ask only how the quantities appearing in \eqref{eq:mid}--\eqref{eq:end} can be supplied and solved at PDE level.

\subsection{Order-two local evolution is enough}
\begin{definition}[ADER2 local evolution]
For a smooth physical quantity \(\Phi(t)\), define
\[
\mathcal E_2[\Phi](\tau)=\Phi(t_n)+\tau\dot\Phi(t_n),\qquad 0\le \tau\le h.
\]
Then \(\Phi(t_n+\tau)-\mathcal E_2[\Phi](\tau)=O(\tau^2)\).
\end{definition}
The adjective ``second order'' refers to this local linear-in-time trajectory, not to a claim that \(\dot\Phi\) itself has a temporal discretization error of order \(h^2\).

For the split field \eqref{eq:split}, the mixed-compatible local models are
\begin{equation}
F(t_n+\tau)=F_n+\tau F_U(F+G)_n+O(\tau^2),\qquad
G(t_n+\tau)=G_n+\tau G_U(F+G)_n+O(\tau^2).
\label{eq:ader2_models}
\end{equation}
The inherited Hermite update needs the midpoint derivative values only with a factor \(h^2\), so an \(O(h^3)\) midpoint state is sufficient.

\begin{proposition}[ADER2 sufficiency for the inherited outer composition]
\label{prop:ader2_suff}
Assume the regularity and local-solvability hypotheses under which the underlying time discretization is fourth-order accurate, and assume the PDE local-evolution module supplies the instantaneous full-field derivatives \eqref{eq:fullfield} at the numerical stages.  If the midpoint stage is solved with \(U^*-U(t_n+h/2)=O(h^3)\), then replacing the exact midpoint by \(U^*\) changes the endpoint Hermite defect by only \(O(h^5)\).  Thus an order-two ADER/CK local trajectory is sufficient; no third- or fourth-order local predictor is required solely to preserve the fourth-order outer method.
\end{proposition}
\begin{proof}
Let $U_m=U(t_n+h/2)$ denote the exact midpoint. The one-sided Hermite formulas used by the inherited outer method can be written, for a smooth scalar or vector quantity $\phi$, as
\[
\int_{t_n}^{t_n+h}\phi(t)\,dt=h\phi(t_n)+\frac{h^2}{6}\bigl(\dot\phi(t_n)+2\dot\phi(U_m)\bigr)+O(h^5),
\]
and, in right-endpoint form,
\[
\int_{t_n}^{t_n+h}\phi(t)\,dt=h\phi(t_n+h)-\frac{h^2}{6}\bigl(\dot\phi(t_n+h)+2\dot\phi(U_m)\bigr)+O(h^5).
\]
Only the midpoint derivative values are affected when $U_m$ is replaced by the numerical stage $U^*$. By hypothesis $U^*-U_m=O(h^3)$. Because the full-field derivative maps are locally Lipschitz on the stage neighborhood, there is a constant $L$, independent of sufficiently small $h$, such that
\[
\|\dot F(U^*)-\dot F(U_m)\|+\|\dot G(U^*)-\dot G(U_m)\|\le L\|U^*-U_m\|=O(h^3).
\]
Each midpoint derivative enters the endpoint update with a factor of order $h^2$. Hence the total change in the endpoint residual caused by using $U^*$ rather than $U_m$ is $O(h^5)$. The endpoint value and endpoint derivative terms are unchanged by this replacement. Therefore the local defect of the inherited fourth-order composition remains of order $h^5$, proving that the order-two local trajectory is sufficient for the outer fourth-order method.
\end{proof}

\section{PDE-level CK construction and reconstruction consistency}
Consider the convection--diffusion--reaction equation
\begin{equation}
u_t+\nabla\cdot f(u)=\nabla\cdot(D\nabla u)+S(u).
\label{eq:cdr}
\end{equation}
Set
\[
F(u)=-\nabla\cdot f(u),\qquad G(u)=\nabla\cdot(D\nabla u)+S(u),\qquad H=F+G.
\]
A first CK differentiation along the complete field gives
\begin{equation}
\dot F=-\nabla\cdot\big(f'(u)H\big),\qquad
\dot G=\nabla\cdot(D\nabla H)+S'(u)H.
\label{eq:ck}
\end{equation}
Only the first physical time derivative is used, so \eqref{eq:ck} remains an ADER2 local model.

After spatial discretization,
\begin{equation}
\frac{dU_h}{dt}=F_h(U_h)+G_h(U_h)=H_h(U_h),
\end{equation}
and the ideal semi-discrete directional derivatives are
\begin{equation}
D_{F,h}=J_{F_h}H_h,\qquad D_{G,h}=J_{G_h}H_h.
\label{eq:jvp}
\end{equation}
For linear operators these are exact matrix products.  For nonlinear WENO calculations we use the PDE-level CK form \eqref{eq:ck}; a centered matrix-free JVP is useful as an implementation cross-check.

\begin{proposition}[Smooth CK reconstruction consistency]
\label{prop:ck_consistency}
Assume \(u\) is smooth and
\[
H_h(\Pi_hu)=\Pi_hH(u)+O(\dx^p)
\]
with a uniformly smooth asymptotic error expansion.  Let \(D_h^{(m)}\) be stable discrete first- and second-derivative operators satisfying
\[
D_h^{(m)}\Pi_h\psi=\Pi_h(\partial_x^m\psi)+O(\dx^{r_m}),\qquad m=1,2.
\]
For
\[
F=-\partial_x f(u),\qquad G=\nu u_{xx}+S(u),
\]
define
\[
\widehat{\dot F}_h=-D_h^{(1)}\big(f'(\Pi_hu)H_h\big),\qquad
\widehat{\dot G}_h=\nu D_h^{(2)}H_h+S'(\Pi_hu)H_h.
\]
Then
\[
\widehat{\dot F}_h=\Pi_h\dot F+O(\dx^{q_F}),\quad q_F=\min\{p,r_1\},
\]
and
\[
\widehat{\dot G}_h=\Pi_h\dot G+O(\dx^{q_G}),\quad q_G=\min\{p,r_2\}.
\]
The multidimensional result follows componentwise.
\end{proposition}
\begin{proof}
Write
\[
H_h(\Pi_hu)=\Pi_hH+\dx^p e_h,
\]
where $e_h$ and the discrete derivatives required below are uniformly bounded. For the convective derivative, add and subtract the exact sampled product to obtain
\[
\widehat{\dot F}_h-\Pi_h\dot F=-\bigl(D_h^{(1)}\Pi_h-\Pi_h\partial_x\bigr)\bigl(f'(u)H\bigr)-D_h^{(1)}\Bigl(f'(\Pi_hu)\bigl(H_h-\Pi_hH\bigr)\Bigr)+\mathcal R_F.
\]
The first term is $O(\dx^{r_1})$ by consistency of $D_h^{(1)}$. The second is $O(\dx^p)$ because the smooth asymptotic expansion of $H_h-\Pi_hH$ is preserved under multiplication by the smooth factor $f'(\Pi_hu)$ and by the assumed stable discrete derivative. The sampling remainder $\mathcal R_F$ has the same or higher order under the stated projection assumptions. Thus the convective derivative error is $O(\dx^{\min(p,r_1)})$.

For the diffusion/source derivative,
\[
\begin{aligned}
\widehat{\dot G}_h-\Pi_h\dot G={}&\nu\bigl(D_h^{(2)}\Pi_h-\Pi_h\partial_{xx}\bigr)H+\nu D_h^{(2)}\bigl(H_h-\Pi_hH\bigr)\\
&+S'(\Pi_hu)\bigl(H_h-\Pi_hH\bigr)+\mathcal R_G.
\end{aligned}
\]
The displayed terms are respectively $O(\dx^{r_2})$, $O(\dx^p)$, $O(\dx^p)$, and a sampling remainder of no lower order, so $q_G=\min\{p,r_2\}$. In several space dimensions the CK formulas are sums of directional terms. Applying the same estimate componentwise and summing proves the multidimensional statement.
\end{proof}

\section{What fails if the split components are differentiated separately}
The companion time-discretization study observed numerically that omitting the mixed derivatives destroys high order for a noncommuting split. Here we identify the leading defect explicitly at the PDE/semi-discrete level.

\begin{proposition}[Leading defect of self differentiation]
\label{prop:commutator}
Replace the full derivatives by the incomplete self derivatives
\[
\widetilde{\dot F}=F_UF,\qquad \widetilde{\dot G}=G_UG.
\]
Then one numerical step of \eqref{eq:mid}--\eqref{eq:end} satisfies
\begin{equation}
\widetilde U^{n+1}=U^n+hH+\frac{h^2}{2}H_UH
+\frac{h^2}{2}(G_UF-F_UG)+O(h^3).
\label{eq:comm_defect}
\end{equation}
For \(F(U)=AU\), \(G(U)=BU\),
\begin{equation}
\widetilde U^{n+1}-U(t_n+h)
=-\frac{h^2}{2}[A,B]U^n+O(h^3),\qquad [A,B]=AB-BA.
\label{eq:linear_comm}
\end{equation}
Thus a nonzero commutator generically produces first-order global accuracy under stability.
\end{proposition}
\begin{proof}
All vector fields and Jacobians in the expansion are evaluated at $U^n$ unless another argument is displayed. Let $\delta^*=U^*-U^n$. The self-differentiated midpoint equation gives
\[
\delta^*=\frac h2F+\frac{h^2}{8}F_UF+\frac h2\bigl(G+G_U\delta^*\bigr)-\frac{h^2}{8}G_UG+O(h^3).
\]
At first order, $\delta^*=\tfrac h2H+O(h^2)$. Substitution back into the $G_U\delta^*$ term gives
\[
U^*=U^n+\frac h2H+\frac{h^2}{8}F_UF+\frac{h^2}{4}G_UF+\frac{h^2}{8}G_UG+O(h^3).
\]
Only the first-order part of this stage expansion is needed in the endpoint derivative terms because those terms already carry a factor $h^2$.

Let $\delta=\widetilde U^{n+1}-U^n$. The endpoint equation implies $\delta=hH+O(h^2)$ and therefore
\[
G(U^n+\delta)=G+G_U\delta+O(h^2)=G+hG_UH+O(h^2).
\]
Moreover, $(F_UF)(U^*)=F_UF+O(h)$, $(G_UG)(U^*)=G_UG+O(h)$, and $(G_UG)(U^n+\delta)=G_UG+O(h)$. Inserting these expansions into the endpoint formula and retaining terms through order $h^2$ yields
\[
\widetilde U^{n+1}=U^n+hH+h^2\left(\frac12F_UF+G_UF+\frac12G_UG\right)+O(h^3).
\]
The exact Taylor expansion is
\[
U(t_n+h)=U^n+hH+\frac{h^2}{2}\bigl(F_UF+F_UG+G_UF+G_UG\bigr)+O(h^3).
\]
Subtracting gives
\[
\widetilde U^{n+1}-U(t_n+h)=\frac{h^2}{2}(G_UF-F_UG)+O(h^3),
\]
which proves \eqref{eq:comm_defect}. For linear fields, $F_U=A$ and $G_U=B$, hence $G_UF-F_UG=(BA-AB)U=-[A,B]U$, proving \eqref{eq:linear_comm}. Under one-step stability, a nonzero local defect of order $h^2$ accumulates to a generic first-order global error.
\end{proof}

\begin{remark}
The defect in \eqref{eq:comm_defect} is the Lie bracket of the two vector fields up to sign convention.  A commuting linear split eliminates the displayed leading term, but that special cancellation is not a general justification for self differentiation in nonlinear or higher-order settings.
\end{remark}

\section{Fully discrete and inexact-stage error budgets}
Let \(\Pi_hu\) denote grid projection and assume
\begin{equation}
F_h(\Pi_hu)=\Pi_hF(u)+O(\dx^p),\qquad
G_h(\Pi_hu)=\Pi_hG(u)+O(\dx^p),
\label{eq:spatial_fg}
\end{equation}
while the reconstructed first temporal derivatives satisfy
\begin{equation}
\widehat D_{F,h}(\Pi_hu)=\Pi_h\dot F(u)+O(\dx^q),\qquad
\widehat D_{G,h}(\Pi_hu)=\Pi_h\dot G(u)+O(\dx^q).
\label{eq:spatial_d}
\end{equation}

\begin{corollary}[Fully discrete error balance]
\label{cor:full_discrete}
Assume the inherited fourth-order temporal method is stable for the fixed semi-discretization, the implicit-stage Jacobians have uniformly bounded local inverses, and \eqref{eq:spatial_fg}--\eqref{eq:spatial_d} hold while the exact solution remains smooth.  Then, for \(t_n\le T\),
\begin{equation}
\norm{U_h^n-\Pi_hu(t_n)}\le C_T\big(h^4+\dx^p+h\dx^q\big).
\label{eq:full_error}
\end{equation}
If \(h=O(\dx)\), the convergence order is at least \(\min\{4,p,q+1\}\).  Hence \(p\ge4\) and \(q\ge3\) are sufficient for fourth-order fully discrete convergence.
\end{corollary}
\begin{proof}
Let $\mathcal M_{h,\dx}$ denote one fully discrete step and insert the projected exact state $\Pi_hu(t_n)$ into both stages. There are three consistency contributions. First, the inherited temporal method has local defect $O(h^5)$ when exact semi-discrete operators and full-field derivatives are supplied. Second, every physical operator $F_h$ or $G_h$ is multiplied by a factor of order $h$, so \eqref{eq:spatial_fg} contributes $O(h\dx^p)$. Third, every reconstructed temporal derivative is multiplied by a factor of order $h^2$, so \eqref{eq:spatial_d} contributes $O(h^2\dx^q)$.

The stage equations are implicit, so residual estimates must be converted into state estimates. By the bounded inverse-Jacobian hypothesis, the local inverse function theorem applies uniformly on the stage neighborhood for sufficiently small $h$. Consequently, a residual perturbation of size $\eta$ changes the corresponding stage state by at most $C\eta$. Thus the projected exact solution produces the one-step defect
\[
\tau_n:=\mathcal M_{h,\dx}(\Pi_hu(t_n))-\Pi_hu(t_{n+1}),\qquad \|\tau_n\|\le C\bigl(h^5+h\dx^p+h^2\dx^q\bigr).
\]
Set $e_n=U_h^n-\Pi_hu(t_n)$. One-step stability gives
\[
\|\mathcal M_{h,\dx}(V)-\mathcal M_{h,\dx}(W)\|\le(1+Ch)\|V-W\|,
\]
so
\[
\|e_{n+1}\|\le(1+Ch)\|e_n\|+C\bigl(h^5+h\dx^p+h^2\dx^q\bigr).
\]
Iterating this inequality for $n\le T/h$ and applying the discrete Gronwall lemma yields
\[
\|e_n\|\le C_T\bigl(h^4+\dx^p+h\dx^q\bigr),
\]
which is \eqref{eq:full_error}. If $h=O(\dx)$, the three terms scale as $\dx^4$, $\dx^p$, and $\dx^{q+1}$, respectively, and the stated minimum order follows.
\end{proof}

\begin{remark}[Scope of the fully discrete estimate]
The constant in \eqref{eq:full_error} is not asserted to be uniform with respect to an independent stiffness parameter.  The corollary is a smooth-solution consistency/stability statement for the chosen semi-discretization; it is not an asymptotic-preserving, asymptotically accurate, or stiffness-uniform convergence theorem.
\end{remark}

In practice the two implicit stages are not solved exactly.  The compact method has an unusually useful asymmetry: the midpoint state appears in the endpoint formula only through derivative terms multiplied by \(h^2\), while the endpoint state itself must satisfy the final implicit equation accurately.

\begin{proposition}[Residual tolerances that preserve fourth order]
\label{prop:inexact}
Let \(U^*\) and \(U^{n+1}\) denote the exact roots of the two implicit stage equations for one step, and let \(\widehat U^*\), \(\widehat U^{n+1}\) be inexact stage solutions.  Let \(r_*\) be the \emph{true} residual of the midpoint equation evaluated at \(\widehat U^*\), and let \(r_1\) be the true residual of the endpoint equation evaluated at \(\widehat U^{n+1}\) with the already computed \(\widehat U^*\).  Assume local inverse stage Jacobians are uniformly bounded, the derivative maps are locally Lipschitz, and the inexact one-step map is stable on the solution neighborhood.  Then the one-step perturbation due to the inexact solves satisfies
\begin{equation}
\norm{\widehat U^{n+1}-U^{n+1}}
\le C\big(h^2\norm{r_*}+\norm{r_1}\big).
\label{eq:inexact_bound}
\end{equation}
Consequently, if
\begin{equation}
\norm{r_*}=O(h^3),\qquad \norm{r_1}=O(h^5),
\label{eq:tols}
\end{equation}
the inexact solves do not change the fourth-order global accuracy.  More generally, residual orders \(O(h^{p_*})\) and \(O(h^{p_1})\) contribute global errors \(O(h^{p_*+1})\) and \(O(h^{p_1-1})\), respectively.
\end{proposition}
\begin{proof}
Let $\mathcal R_*(V)$ denote the exact midpoint-stage residual and let $\mathcal R_1(W;V)$ denote the endpoint-stage residual, with the second argument recording the midpoint state used in the endpoint formula. Then $\mathcal R_*(U^*)=0$ and $\mathcal R_1(U^{n+1};U^*)=0$, while $r_*=\mathcal R_*(\widehat U^*)$ and $r_1=\mathcal R_1(\widehat U^{n+1};\widehat U^*)$.

By the bounded inverse midpoint Jacobian and a mean-value expansion of $\mathcal R_*$,
\[
\|\widehat U^*-U^*\|\le C\|r_*\|.
\]
The midpoint state enters the endpoint residual only through derivative terms multiplied by $h^2$. Therefore, using the local Lipschitz property of the derivative maps,
\[
\|\mathcal R_1(U^{n+1};\widehat U^*)-\mathcal R_1(U^{n+1};U^*)\|\le Ch^2\|\widehat U^*-U^*\|\le Ch^2\|r_*\|.
\]
Add and subtract $\mathcal R_1(U^{n+1};\widehat U^*)$ in the residual at $\widehat U^{n+1}$. The direct endpoint residual is $r_1$, while the additional residual induced by the inexact midpoint is bounded by the preceding estimate. Applying the bounded inverse endpoint Jacobian gives
\[
\|\widehat U^{n+1}-U^{n+1}\|\le C\bigl(h^2\|r_*\|+\|r_1\|\bigr),
\]
which is \eqref{eq:inexact_bound}.

If $\|r_*\|=O(h^{p_*})$ and $\|r_1\|=O(h^{p_1})$, the local perturbation is $O(h^{p_*+2})+O(h^{p_1})$. Stable accumulation over $O(h^{-1})$ steps yields global contributions $O(h^{p_*+1})$ and $O(h^{p_1-1})$. Taking $p_*=3$ and $p_1=5$ therefore preserves fourth-order convergence.
\end{proof}

\begin{remark}[Stopping tests]
The powers in \eqref{eq:tols} refer to the true stage-equation residual in a fixed norm (or to a consistently scaled dimensionless version of that residual).  A Krylov or nonlinear solver should therefore not infer the required order solely from an internal step-size or preconditioned-residual test unless that test is known to control the true residual.
\end{remark}

\section{From mixed compatibility to widened implicit stages}
The pure-implicit stability function and its quadratic stiff decay were derived in the companion time-discretization manuscript; we do not repeat that stability theorem here.  We only recall the asymptotic form
\begin{equation}
R_I(z)=-10z^{-2}+O(z^{-3}),\qquad |z|\to\infty,
\label{eq:ri_recall}
\end{equation}
because the same quadratic stiff polynomial appears in the PDE stage operator and motivates the solver design below.

For \(F_h=A_hU\), \(G_h=B_hU\), both implicit stages can be written
\begin{equation}
M_j=I-\alpha_jhB_h+\beta_jh^2B_h(A_h+B_h),
\qquad
(\alpha_1,\beta_1)=\left(\frac12,\frac18\right),\quad
(\alpha_2,\beta_2)=\left(1,\frac16\right).
\label{eq:mixed_stage}
\end{equation}
The exact action is applied matrix-free,
\begin{equation}
M_jv=v-\alpha_jhB_hv+\beta_jh^2B_h(A_hv+B_hv),
\label{eq:matrixfree}
\end{equation}
so neither \(B_hA_h\) nor \(B_h^2\) needs to be assembled.

\section{Structure-aware implicit solvers}
\subsection{Compact and quadratic stiff-physics preconditioners}
The lowest-cost preconditioner is
\[
P_j^{(1)}=I-\alpha_jhB_h.
\]
When the quadratic stiff term dominates, use
\begin{equation}
P_j^{(2)}=I-\alpha_jhB_h+\beta_jh^2B_h^2.
\label{eq:quadratic_precond}
\end{equation}
The polynomial can be realized through shifted factors
\[
P_j^{(2)}=(I-r_{j,+}hB_h)(I-r_{j,-}hB_h),\qquad
r_{j,\pm}=\frac{\alpha_j\pm\sqrt{\alpha_j^2-4\beta_j}}{2},
\]
so \(B_h^2\) need not be stored.

\begin{proposition}[Exact cancellation of the pure stiff quadratic term]
\label{prop:quadratic_cancel}
If \(P_j^{(2)}\) is nonsingular, then
\begin{equation}
(P_j^{(2)})^{-1}M_j
=I+\beta_jh^2(P_j^{(2)})^{-1}B_hA_h.
\label{eq:preconditioned}
\end{equation}
Thus the complete \(B_h^2\) contribution is removed algebraically, independently of whether \(A_h\) and \(B_h\) commute.
\end{proposition}
\begin{proof}
Expanding the mixed stage operator gives
\[
M_j=I-\alpha_jhB_h+\beta_jh^2B_h^2+\beta_jh^2B_hA_h.
\]
The first three terms are exactly $P_j^{(2)}$, so
\[
M_j=P_j^{(2)}+\beta_jh^2B_hA_h.
\]
Since $P_j^{(2)}$ is nonsingular, left multiplication by its inverse gives \eqref{eq:preconditioned}. No interchange of $A_h$ and $B_h$ is used, so the cancellation of the pure stiff contribution does not require commutativity. The remaining perturbation is exactly the transport-mediated mixed action.
\end{proof}

The identity is structural rather than a general mesh-independent GMRES theorem.  In the common periodic constant-coefficient case, however, a sharper modewise statement is available.

\begin{proposition}[Diffusion-dominated Fourier clustering]
\label{prop:fourier}
Assume \(A_h\) and \(B_h\) are simultaneously diagonalized by the Fourier basis, with mode symbols \(a_k\) and \(b_k\le0\).  For the quadratic preconditioner \eqref{eq:quadratic_precond}, the preconditioned correction on mode \(k\) is
\[
E_{j,k}=\frac{\beta_jh^2b_ka_k}{1-\alpha_jhb_k+\beta_jh^2b_k^2}.
\]
Then
\begin{equation}
|E_{j,k}|\le \frac{|a_k|}{|b_k|}\qquad (b_k\ne0).
\label{eq:fourier_bound}
\end{equation}
For constant-coefficient advection--diffusion, \(a_k=i c\xi_k\), \(b_k=-\nu\xi_k^2\), hence
\[
|E_{j,k}|\le \frac{|c|}{\nu|\xi_k|}.
\]
Therefore diffusion-dominated high-frequency modes cluster toward the identity as \(|\xi_k|\to\infty\), uniformly with respect to the magnitude of \(h|b_k|\).
\end{proposition}
\begin{proof}
On Fourier mode $k$, the stage and quadratic preconditioner reduce to scalar multipliers. Proposition~\ref{prop:quadratic_cancel} therefore gives
\[
E_{j,k}=\frac{\beta_jh^2b_ka_k}{1-\alpha_jhb_k+\beta_jh^2b_k^2}.
\]
Because $b_k\le0$, the denominator is the positive real number
\[
1+\alpha_jh|b_k|+\beta_jh^2|b_k|^2,
\]
which is bounded below by $\beta_jh^2|b_k|^2$. Hence, for $b_k\ne0$,
\[
|E_{j,k}|\le\frac{\beta_jh^2|b_k||a_k|}{\beta_jh^2|b_k|^2}=\frac{|a_k|}{|b_k|}.
\]
For constant-coefficient advection--diffusion, $|a_k|$ scales linearly with the wave number while $|b_k|$ scales quadratically. Thus the ratio decays inversely with frequency, and the preconditioned multiplier approaches one for diffusion-dominated high-frequency modes without requiring the stiff number $h|b_k|$ to be small.
\end{proof}

\subsection{Semilinear reaction--diffusion reduction}
Let
\[
F(U)=R(U),\qquad G(U)=LU,
\]
with linear block diffusion \(L\).  Since \(\dot G=L(R(U)+LU)\), either implicit stage can be written
\begin{equation}
P_jU_j+\beta_jh^2LR(U_j)=r_j,
\qquad P_j=I-\alpha_jhL+\beta_jh^2L^2.
\label{eq:semilinear_stage}
\end{equation}
The nonlinear wide Jacobian is avoided by the fixed-point map
\begin{equation}
\mathcal T_j(V)=P_j^{-1}\big(r_j-\beta_jh^2LR(V)\big).
\end{equation}

\begin{proposition}[Reaction--diffusion contraction]
If \(R\) is Lipschitz with constant \(L_R\) on a convex stage neighborhood and
\[
\kappa_j=\beta_jh^2\norm{P_j^{-1}L}L_R<1,
\]
then \(\mathcal T_j\) is a contraction and the Picard iterates converge geometrically.  In particular,
\[
\norm{V^{(m)}-U_j}\le \frac{\kappa_j^m}{1-\kappa_j}\norm{V^{(1)}-V^{(0)}}.
\]
\end{proposition}
\begin{proof}
For any $V$ and $W$ in the convex stage neighborhood,
\[
\mathcal T_j(V)-\mathcal T_j(W)=-\beta_jh^2P_j^{-1}L\bigl(R(V)-R(W)\bigr).
\]
Taking norms and using the Lipschitz property of $R$ gives
\[
\|\mathcal T_j(V)-\mathcal T_j(W)\|\le\beta_jh^2\|P_j^{-1}L\|L_R\|V-W\|=\kappa_j\|V-W\|.
\]
Thus $\mathcal T_j$ is a contraction when $\kappa_j<1$. The Banach fixed-point theorem gives a unique stage solution in the neighborhood and convergence of the Picard iterates. Moreover,
\[
\|V^{(m)}-U_j\|\le\sum_{\ell=m}^{\infty}\|V^{(\ell+1)}-V^{(\ell)}\|\le\frac{\kappa_j^m}{1-\kappa_j}\|V^{(1)}-V^{(0)}\|,
\]
which proves the stated estimate.
\end{proof}

\subsection{FFT, multilevel, and local-source realizations}
For a periodic constant principal part
\[
B_0=\nu(D_{xx,h}+D_{yy,h})-\kappa_0I,
\]
we use
\[
P_{j,0}^{\mathrm{FFT}}=I-\alpha_jhB_0+\beta_jh^2B_0^2,
\]
which is inverted exactly in Fourier space.  Variable reaction and noncommuting transport remain in the exact matrix-free residual.  Stage-specific LGMRES augmentation vectors are recycled across time steps~\cite{baker2005,parks2006}.  When Krylov wall-clock comparisons are considered, the reference method should be given the same principal-physics preconditioner and recycling strategy.

For nonperiodic problems, a geometric/Galerkin multilevel approximation is applied only to a reference principal operator \(B_0\); the exact stage action still contains the full variable-coefficient \(B_h\) and mixed product \(B_h(A_h+B_h)\).  A representative geometric realization uses bilinear prolongation, full-weighting restriction, Galerkin coarse operators, weighted Jacobi smoothing, and a direct coarse solve, following standard multigrid principles~\cite{trottenberg2001}. This construction is included to show portability of the separation principle; no claim of a production AMG implementation is made.

If \(B_h=\mathrm{diag}(B_1,\dots,B_N)\) is a cell-local stiff source, define
\[
P_j^{\mathrm{blk}}=\mathrm{diag}\big(I-\alpha_jhB_i+\beta_jh^2B_i^2\big)_{i=1}^N.
\]
For a pure source, the preconditioned stage is exactly the identity, independently of the stiffness magnitude.

\subsection{Jin--Xin relaxation and finite Neumann correction}
Consider the Jin--Xin model~\cite{jinxin1995}
\begin{equation}
u_t+v_x=0,\qquad v_t+a^2u_x=\frac{f(u)-v}{\varepsilon},\qquad f(u)=\frac12u^2.
\label{eq:jinxin}
\end{equation}
At an implicit stage, the \(u\) component is fixed by the stage constant.  The \(v\) equation can be written
\[
\big(c_jI-\gamma_j\operatorname{diag}(f'(u))D_h\big)v=\mathrm{rhs},
\]
where
\[
c_j=1+\frac{\alpha_jh}{\varepsilon}+\frac{\beta_jh^2}{\varepsilon^2},
\qquad
\gamma_j=\frac{\beta_jh^2}{\varepsilon}.
\]

\begin{theorem}[Stiff-limit clustering after source elimination]
\label{thm:jinxin}
Preconditioning by the exact local source block \(c_jI\) gives
\begin{equation}
P_j^{-1}M_{v,j}=I-\theta_j(\varepsilon)\operatorname{diag}(f'(u))D_h,
\quad
\theta_j(\varepsilon)=\frac{\beta_jh^2/\varepsilon}{1+\alpha_jh/\varepsilon+\beta_jh^2/\varepsilon^2},
\end{equation}
with
\begin{equation}
0\le \theta_j(\varepsilon)\le \varepsilon.
\end{equation}
Hence, if \(\norm{D_h}\le C_D/\dx\),
\begin{equation}
\norm{P_j^{-1}M_{v,j}-I}
\le C_D\norm{f'(u)}_\infty\frac{\varepsilon}{\dx}.
\label{eq:epsdx}
\end{equation}
The family is uniformly bounded when \(\varepsilon/\dx=O(1)\) and clusters to the identity when \(\varepsilon/\dx\to0\).
\end{theorem}
\begin{proof}
Dividing the stage equation by the positive scalar $c_j$ gives
\[
P_j^{-1}M_{v,j}=I-\frac{\gamma_j}{c_j}\operatorname{diag}(f'(u))D_h.
\]
The scalar coefficient is
\[
\theta_j(\varepsilon)=\frac{\gamma_j}{c_j}=\varepsilon\frac{\beta_jh^2}{\varepsilon^2+\alpha_jh\varepsilon+\beta_jh^2}.
\]
For both compact stages, $\alpha_j>0$ and $\beta_j>0$. The denominator is therefore at least $\beta_jh^2$, so the fraction lies in $[0,1]$ and $0\le\theta_j(\varepsilon)\le\varepsilon$. Consequently,
\[
\begin{aligned}
\|P_j^{-1}M_{v,j}-I\|&\le\theta_j(\varepsilon)\|\operatorname{diag}(f'(u))D_h\|\\
&\le\varepsilon\|f'(u)\|_\infty\|D_h\|\\
&\le C_D\|f'(u)\|_\infty\frac{\varepsilon}{\dx},
\end{aligned}
\]
which proves \eqref{eq:epsdx}. If $\varepsilon/\dx$ remains bounded, the perturbation remains bounded independently of the stiffness magnitude; if $\varepsilon/\dx\to0$, it vanishes and the preconditioned stage clusters to the identity.
\end{proof}

\begin{remark}
Theorem~\ref{thm:jinxin} is a stage-solver clustering result.  It does not by itself establish asymptotic preservation or asymptotic accuracy of the fully discrete relaxation solution as \(\varepsilon\to0\).
\end{remark}

After division by $c_j$, write $(I-K_j)v=b_j$, where $\norm{K_j}=O(\varepsilon/\dx)$.

\begin{proposition}[Finite Neumann correction]
\label{prop:neumann}
Assume $\rho_j:=\norm{K_j}<1$ and initialize $v^{[0]}=b_j$. For
\[
v^{[\ell+1]}=b_j+K_jv^{[\ell]},\qquad \ell=0,\ldots,m-1,
\]
one has
\[
v^{[m]}=\sum_{\ell=0}^mK_j^\ell b_j,\qquad
\norm{v-v^{[m]}}\le\frac{\rho_j^{m+1}}{1-\rho_j}\norm{b_j}.
\]
Hence, for fixed $m$, the stage-solve error is $O((\varepsilon/\dx)^{m+1})$ as $\varepsilon/\dx\to0$.
\end{proposition}
\begin{proof}
The iteration formula gives the partial Neumann sum by induction. Indeed, $v^{[0]}=b_j$, and if $v^{[m-1]}=\sum_{\ell=0}^{m-1}K_j^\ell b_j$, then
\[
v^{[m]}=b_j+K_jv^{[m-1]}=\sum_{\ell=0}^{m}K_j^\ell b_j.
\]
Because $\rho_j<1$, the exact solution has the convergent representation
\[
v=(I-K_j)^{-1}b_j=\sum_{\ell=0}^{\infty}K_j^\ell b_j.
\]
Subtracting the finite sum and applying the geometric-series bound gives
\[
\norm{v-v^{[m]}}\le\sum_{\ell=m+1}^{\infty}\norm{K_j}^\ell\norm{b_j}=\frac{\rho_j^{m+1}}{1-\rho_j}\norm{b_j}.
\]
Theorem~\ref{thm:jinxin} gives $\rho_j=O(\varepsilon/\dx)$ in the source-dominated regime, which yields the stated asymptotic order.
\end{proof}

\section{Numerical experiments}
The experiments are deliberately separated into (i) mechanism checks, (ii) literature-standard PDE benchmarks, and (iii) solver-mechanism checks tied directly to the preconditioning theory. Timings are medians of repeated runs in the same Python environment and are used only for within-experiment comparisons. The reference implementation and numerical data used for the reproducible mechanism and Brusselator studies are publicly archived in the Zenodo repository identified in the Data Availability statement. The goal is not to claim a universal speedup.

\subsection{Mechanism checks and the inexact-stage budget}
For smooth one- and two-dimensional Burgers--reaction problems, the implemented WENO-Z5/CK derivatives approach fifth-order spatial consistency, safely above the \(q\ge3\) threshold of Corollary~\ref{cor:full_discrete}.  The ADER2 local predictor approaches order two while the inherited outer composition remains fourth order.  In noncommuting linear and nonlinear WENO ablations, full differentiation is fourth order whereas self differentiation is first order, in agreement with Proposition~\ref{prop:commutator}.

\begin{figure}[!htbp]
\centering
\includegraphics[width=0.95\textwidth]{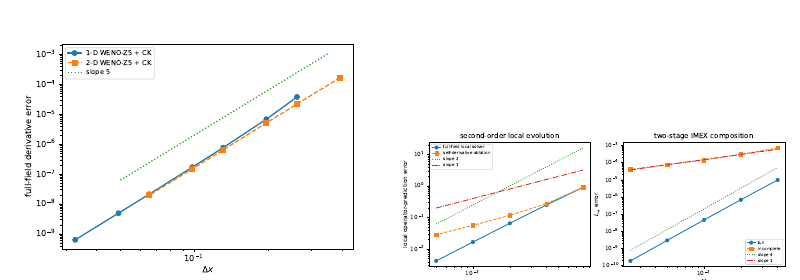}
\caption{Direct CK derivative consistency and the ADER2/local-composition mechanism.  The local trajectory is order two, while the inherited outer method is fourth order when the full-field derivatives are supplied.}
\label{fig:ader2}
\end{figure}

\begin{figure}[!htbp]
\centering
\includegraphics[width=0.90\textwidth]{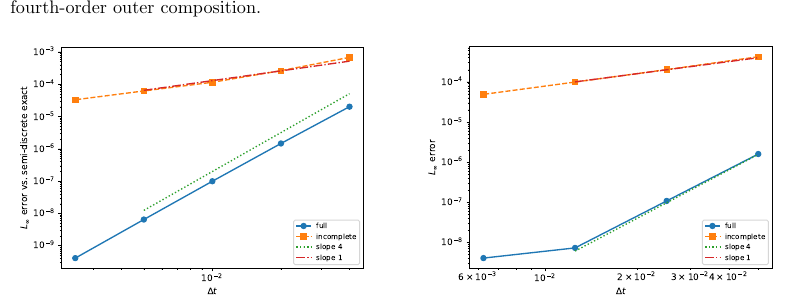}
\caption{Mixed-compatibility ablations.  Full-field differentiation is fourth order, while self differentiation is approximately first order for a noncommuting linear split and a nonlinear WENO realization.}
\label{fig:commutator}
\end{figure}

To verify Proposition~\ref{prop:inexact} independently of nonlinear solver idiosyncrasies, we inject residuals of controlled size into both linear stage equations.  With \(\norm{r_*}=O(h^3)\) and \(\norm{r_1}=O(h^5)\), fourth-order convergence is retained; making both residuals one order looser gives approximately third-order convergence.  This experiment is produced directly by the revised verification code using controlled residual injection, so its observed orders are independent of the nonlinear solver implementation.

\begin{figure}[!htbp]
\centering
\includegraphics[width=0.58\textwidth]{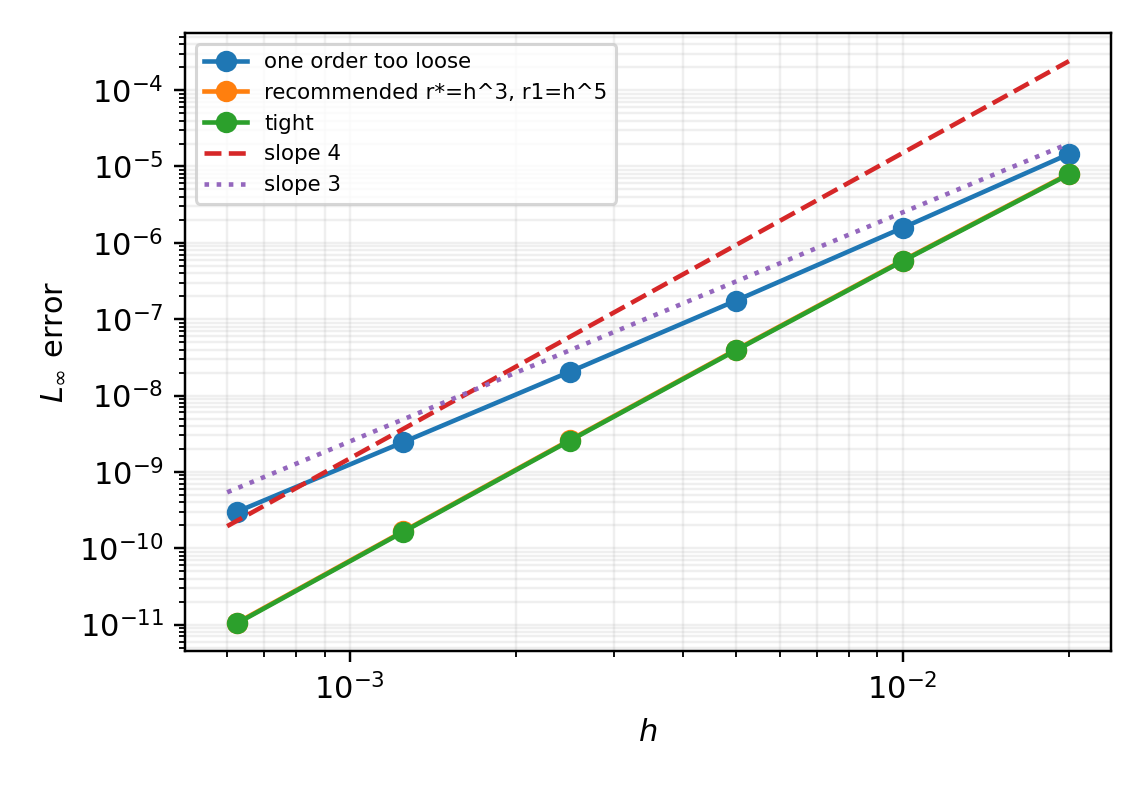}
\caption{Controlled inexact-stage experiment.  The recommended residual scales \(r_*=O(h^3)\), \(r_1=O(h^5)\) preserve fourth-order convergence, while the one-order-looser pair produces third-order behavior.}
\label{fig:tolerance}
\end{figure}

\subsection{Classical stiff front: LeVeque--Yee}
We use the classical scalar balance law~\cite{leveque1990,wang2012}
\begin{equation}
u_t+u_x=-\mu u\left(u-\frac12\right)(u-1),
\qquad
u(x,0)=\begin{cases}1,&x\le0.3,\\0,&x>0.3.\end{cases}
\label{eq:leveque}
\end{equation}
The exact discontinuity translates at unit speed.  With \(N=150\), CFL \(=0.4\), and \(T=0.3\), both methods locate the front well at \(\mu=100\), while both exhibit the classical wrong-speed pathology at \(\mu=1000\).  The measured \(u=1/2\) crossings are shown in Table~\ref{tab:leveque}.  This negative result is important: mixed-compatible time differentiation does not replace shock-aware source reconstruction.

\begin{table}[!htbp]
\centering
\caption{LeVeque--Yee front location.  Error is measured in grid cells relative to the exact crossing \(x=0.6\).}
\label{tab:leveque}
\begin{tabular}{cccc}
\toprule
\(\mu\) & method & front location & error/cell\\
\midrule
100 & Compact-ADER2 & 0.59818 & -0.27\\
100 & KC-ARK4 & 0.59969 & -0.05\\
1000 & Compact-ADER2 & 0.56200 & -5.70\\
1000 & KC-ARK4 & 0.57093 & -4.36\\
\bottomrule
\end{tabular}
\end{table}

\begin{figure}[!htbp]
\centering
\includegraphics[width=0.76\textwidth]{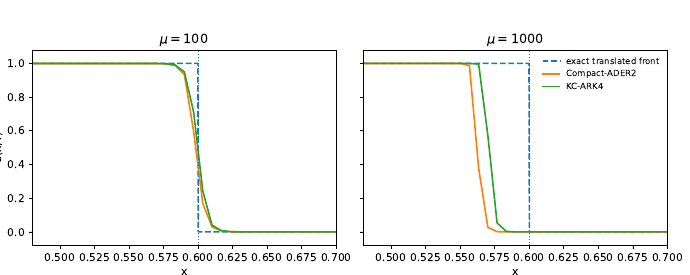}
\caption{LeVeque--Yee stiff-front benchmark at \(T=0.3\).  Moderate stiffness is well located; the extreme case exposes the underresolved shock--source propagation error.}
\label{fig:leveque}
\end{figure}

\subsection{Two-dimensional Brusselator: matched work, Pareto behavior, and a larger grid}
We use the classical reaction--diffusion system in the benchmark configuration of~\cite{asante2016}:
\begin{align}
u_t&=D_u\Delta u+A+u^2v-(B+1)u,\\
v_t&=D_v\Delta v+Bu-u^2v,
\end{align}
with homogeneous Neumann boundary conditions on $[0,1]^2$, $A=1$, $B=3.4$, $D_u=D_v=2\times10^{-3}$, and initial data $u=1/2+y$, $v=1+5x$. Reaction is explicit and diffusion is implicit. For this benchmark the reproducible reference implementation uses centered second differences with ghost reflection for the Neumann Laplacian. Because all temporal comparisons use the same spatial grid and a temporally refined KC-ARK4 reference on that grid, the reported errors measure the time-integration and stage-solve effects rather than spatial convergence.

The compact stages use the polynomially preconditioned fixed-point reduction with two applications of the polynomial inverse per stage: the first application initializes the fixed-point iterate and the second performs one Picard correction. Each polynomial inverse application counts as one solve unit. The DCT realization diagonalizes the Neumann Laplacian and applies the quadratic inverse mode by mode, so neither $L^2$ nor a widened sparse matrix is assembled in this experiment. KC-ARK4 uses five diffusion solves per time step. All timings below are medians of repeated runs in the same Python process and are intended only for within-experiment comparison.

At the matched budget of 200 solve units, Table~\ref{tab:brusselator} shows results on both $41^2$ and $81^2$ grids. On $41^2$, Compact-ADER2 reduces the $L^\infty$ temporal/stage error by a factor of about $3.43$ and the RMS error by a factor of about $2.89$ relative to KC-ARK4, while the median wall time is about $14\%$ lower. On $81^2$, the corresponding factors are about $3.37$ and $2.98$, with the compact wall time about $13\%$ lower. The coarse/fine reference differences are $1.31\times10^{-8}$ and $1.64\times10^{-6}$, respectively, well below the errors reported in the table.

\begin{table}[!htbp]
\centering
\caption{Two-dimensional Brusselator at $T=5$ with a matched budget of 200 polynomial/diffusion solve units. Wall time is the median measured in the reproducible reference implementation.}
\label{tab:brusselator}
\begin{tabular}{c l c c c c c}
\toprule
grid & method & $\dt$ & steps & time (s) & $L^\infty$ error & RMS error\\
\midrule
$41^2$ & Compact-ADER2 & 0.100 & 50 & 0.152 & $2.40\times10^{-3}$ & $2.77\times10^{-4}$\\
$41^2$ & KC-ARK4       & 0.125 & 40 & 0.177 & $8.22\times10^{-3}$ & $8.00\times10^{-4}$\\
$81^2$ & Compact-ADER2 & 0.100 & 50 & 0.351 & $2.83\times10^{-3}$ & $2.93\times10^{-4}$\\
$81^2$ & KC-ARK4       & 0.125 & 40 & 0.402 & $9.53\times10^{-3}$ & $8.73\times10^{-4}$\\
\bottomrule
\end{tabular}
\end{table}

A single matched-work point can hide a different conclusion at other tolerances, so we additionally sweep the step count and plot error against measured wall time. Figure~\ref{fig:brusselator_pareto} shows that the compact method remains on the favorable side of the wall-time/error tradeoff over a substantial part of the tested range on both grids, although the curves approach one another at the smallest errors. This is stronger evidence than a solve-count comparison alone, but it still supports a regime-dependent rather than universal efficiency claim.

\begin{figure}[!htbp]
\centering
\begin{subfigure}[t]{0.48\textwidth}
\centering
\includegraphics[width=\textwidth]{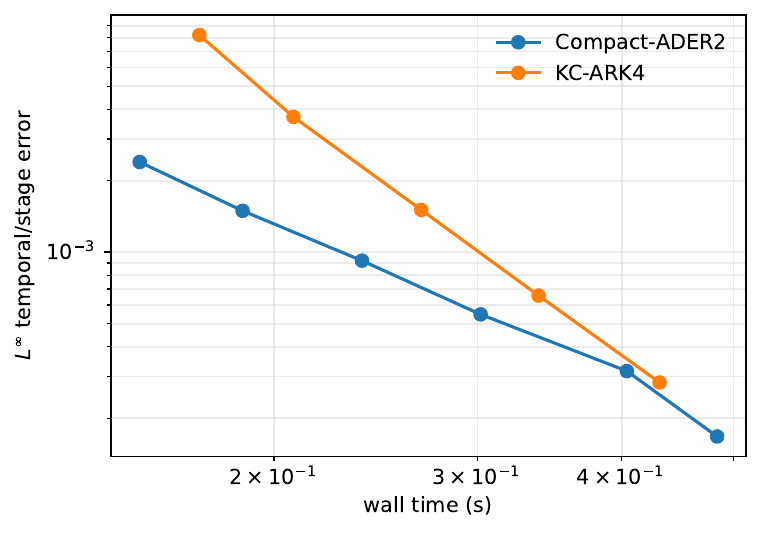}
\caption{$41^2$ grid.}
\end{subfigure}\hfill
\begin{subfigure}[t]{0.48\textwidth}
\centering
\includegraphics[width=\textwidth]{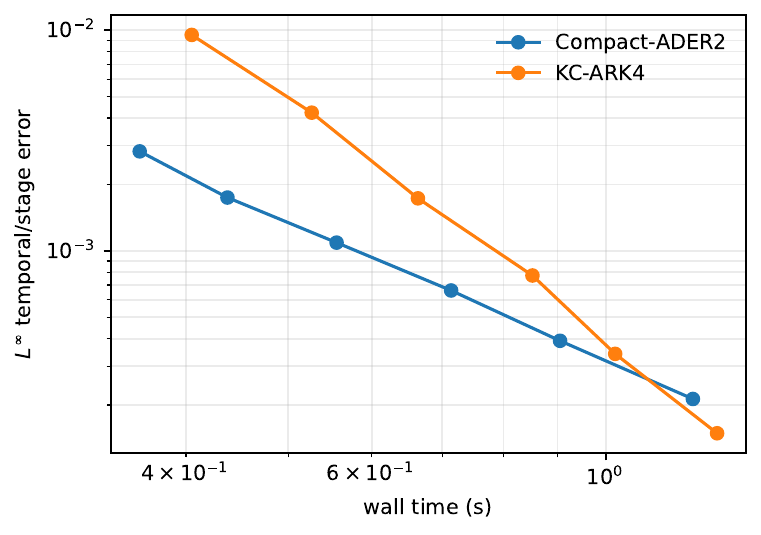}
\caption{$81^2$ grid.}
\end{subfigure}
\caption{Brusselator $L^\infty$ temporal/stage error versus measured wall time. Only stable step counts are plotted. The two-grid sweep tests whether the work-normalized advantage survives a direct time-to-accuracy comparison.}
\label{fig:brusselator_pareto}
\end{figure}

\begin{figure}[!htbp]
\centering
\includegraphics[width=0.90\textwidth]{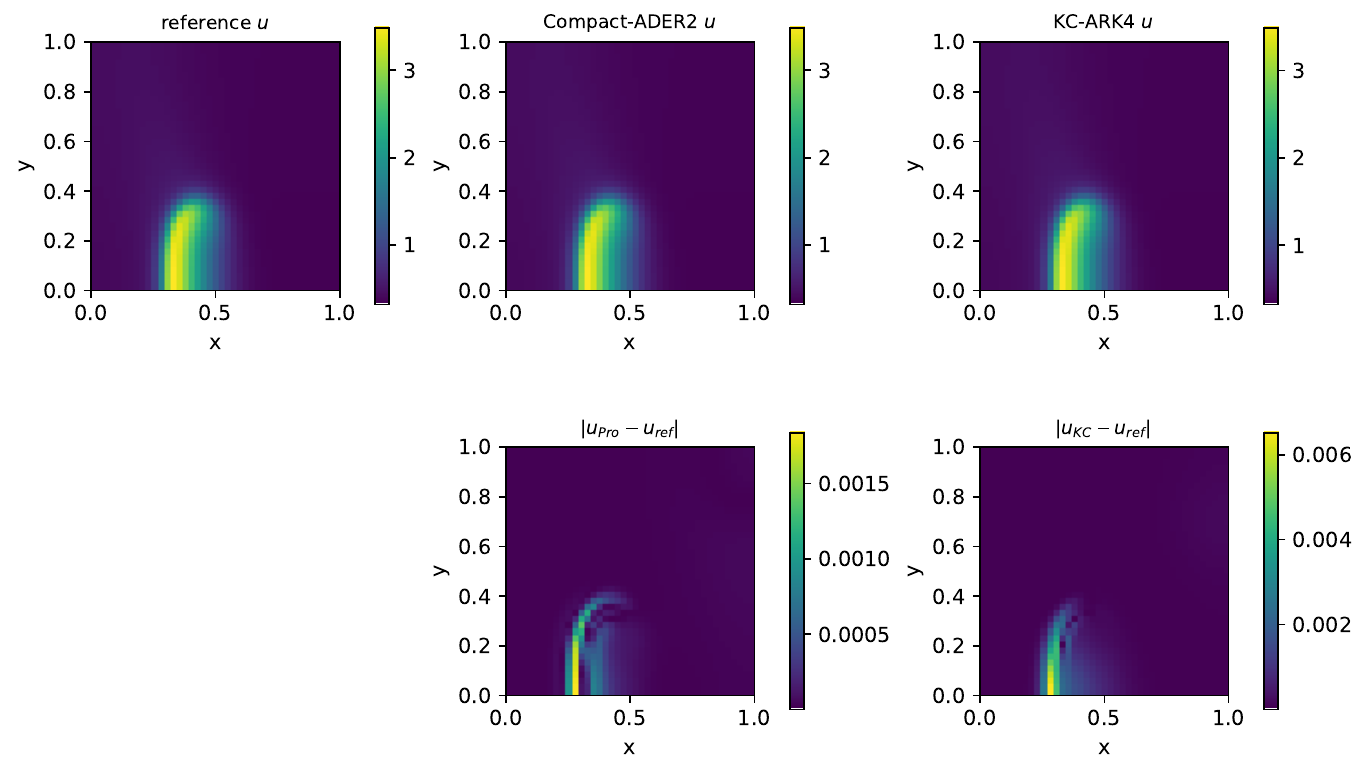}
\caption{Brusselator $u$ field on the $41^2$ grid at $T=5$: temporally refined reference, matched-work Compact-ADER2 and KC-ARK4 solutions, and their pointwise errors.}
\label{fig:brusselator}
\end{figure}

\subsection{Reproducible solver-mechanism checks}
The remaining experiments target the two solver estimates that are most specific to the PDE realization.  First, for constant-coefficient advection--diffusion we evaluate the scalar correction in Proposition~\ref{prop:fourier} over Fourier modes.  Figure~\ref{fig:solver_mechanisms}(a) shows that both compact stages remain below the modewise bound $|c|/(\nu |\xi_k|)$ and that the correction decays with frequency.  This verifies that the quadratic inverse removes the dominant diffusive stiffness while leaving only a transport-mediated perturbation.

Second, we test the finite Neumann approximation after exact source elimination in the Jin--Xin stage.  A periodic centered derivative is used on a fixed grid, $h=0.2\Delta x$, and the ratio $\varepsilon/\Delta x$ is decreased.  Figure~\ref{fig:solver_mechanisms}(b) reports the relative stage error after zero, one, and two Neumann corrections.  Least-squares slopes over the six smallest ratios are approximately $1.0$, $2.0$, and $3.0$, respectively, in agreement with Proposition~\ref{prop:neumann}.  These tests isolate the solver mechanisms without conflating them with the temporal error constants of a full relaxation calculation.

\begin{figure}[!htbp]
\centering
\begin{subfigure}[t]{0.48\textwidth}
\centering
\includegraphics[width=\textwidth]{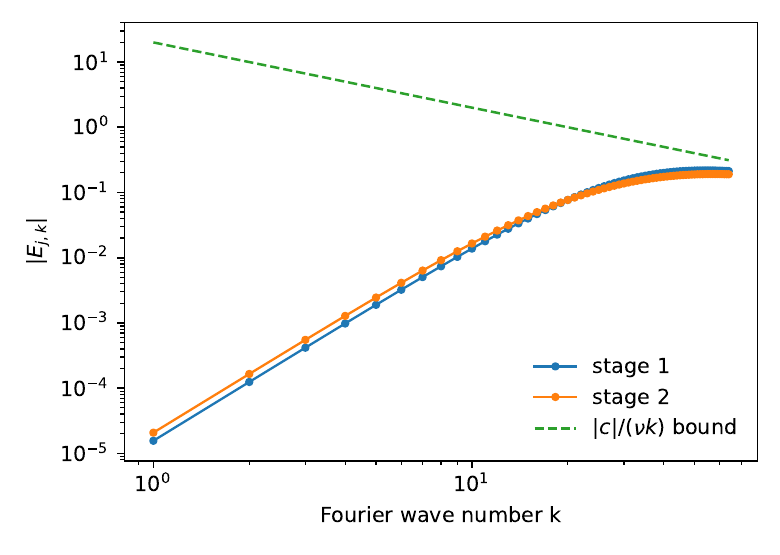}
\caption{Quadratic-preconditioned Fourier correction.}
\end{subfigure}\hfill
\begin{subfigure}[t]{0.48\textwidth}
\centering
\includegraphics[width=\textwidth]{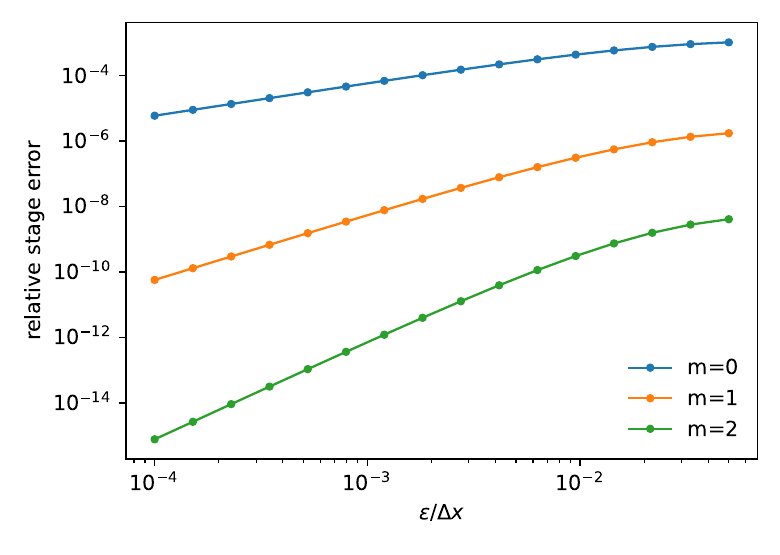}
\caption{Finite Neumann stage error for Jin--Xin relaxation.}
\end{subfigure}
\caption{Reproducible checks of the two main solver estimates.  The left panel verifies the diffusion-dominated Fourier bound; the right panel verifies the predicted order of the finite Neumann correction as $\varepsilon/\Delta x\to0$.}
\label{fig:solver_mechanisms}
\end{figure}

The multilevel and local-source constructions in Section~7 remain useful realizations of the same exact-residual/approximate-inverse principle, but no general mesh-independent GMRES theorem is claimed for variable-coefficient nonnormal operators.  The practical evidence in this paper is therefore concentrated on the reproducible Brusselator time-to-accuracy study and on mechanism tests that directly match the proved estimates.

\section{Discussion}
The paper is organized around one consistency--cost chain rather than around the individual solver technologies.  The full-field cross interaction is required by the PDE realization, that same interaction widens the implicit stage, and the solver design is successful only if it reduces the algebraic cost without deleting the term responsible for consistency.  Four consequences of this chain should not be conflated.

First, mixed compatibility as a property of the compact time discretization belongs to the companion study.  The new PDE-level result here is the realization mechanism: full-field CK derivatives preserve the inherited construction, while self differentiation produces the explicit Lie-bracket defect \eqref{eq:comm_defect}.  This distinction also makes clear why ``ADER for stiff PDEs'' is not itself the novelty.

Second, the same mixed term that is needed for accuracy produces the main implicit-solver difficulty.  The solver hierarchy is therefore not a collection of unrelated preconditioners.  Each method is an implementation of one separation principle: keep \(B_h(A_h+B_h)\) in the exact residual while solving only a dominant stiff approximation.  Proposition~\ref{prop:quadratic_cancel} gives the algebraic cancellation, Proposition~\ref{prop:fourier} gives a diffusion-dominated modewise clustering result, the semilinear contraction handles reaction--diffusion, and Theorem~\ref{thm:jinxin} explains why source-dominated relaxation becomes easier after exact local elimination.

Third, an implicit solve count is not a wall-clock metric.  The recomputed Brusselator study therefore reports both matched-work results and direct error-versus-wall-time sweeps on two grids.  In that semilinear setting the compact method retains a favorable time-to-accuracy tradeoff over a useful range.  The Fourier and relaxation experiments separately verify the mechanisms that make the widened stages inexpensive in diffusion- and source-dominated regimes.  The appropriate claim is therefore regime-dependent efficiency rather than universal speedup or global Pareto dominance.

Fourth, time integration cannot fix a spatial shock--source inconsistency that it does not represent.  The LeVeque--Yee result is retained specifically because it prevents an overly broad interpretation of the smooth and finite-amplitude WENO tests.  Shock-aware source reconstruction remains a separate spatial ingredient~\cite{wang2012}.

\section{Conclusions}
Starting from the compact two-stage fourth-order IMEX time discretization of the companion time-discretization manuscript, we have developed a PDE realization and solver architecture whose contributions are distinct from the time-integrator construction itself.  An order-two CK/ADER local evolution is enough to supply the inherited fourth-order composition, provided the derivatives follow the complete semi-discrete field.  If the split components are differentiated separately, the leading error is a Lie-bracket/commutator defect and noncommuting problems generically fall to first order.

At the fully discrete level, the reconstructed derivative order enters through \(h\dx^q\), so only \(q\ge3\) is required when \(h=O(\dx)\).  Inexact implicit stages may also be solved economically: midpoint and endpoint residuals of sizes \(O(h^3)\) and \(O(h^5)\) preserve fourth-order convergence.  This tolerance asymmetry has been added to the revised reference implementation and verified by controlled residual injection.

The same cross interaction that preserves temporal consistency widens the implicit stages.  Matrix-free application keeps that interaction exact, while quadratic/shifted, FFT, multilevel, semilinear, and source-local solvers approximate only the dominant stiff inverse.  The quadratic preconditioner removes the pure \(B_h^2\) term exactly and clusters diffusion-dominated Fourier modes; source elimination in Jin--Xin relaxation leaves an \(O(\varepsilon/\dx)\) transport correction with a controlled finite Neumann approximation.

The numerical evidence supports a precise rather than universal conclusion.  In the recomputed Brusselator study, matched-work accuracy gains persist on both $41^2$ and $81^2$ grids and are accompanied by a favorable wall-time/error tradeoff over the tested range.  Separate Fourier and relaxation experiments confirm the clustering and finite-Neumann mechanisms predicted by the solver analysis.  The method does not remove the classical underresolved shock--source propagation error, and no general mesh-independent Krylov claim is made for heterogeneous nonnormal problems.  These limitations identify shock-aware reconstruction, mixed-aware operator optimization, and broader reference-method comparisons as natural next steps.

\section*{Thanks}
The author gratefully acknowledges the financial support of the Key Program of Henan Higher Education Institutions (Grant No. 26A110007), the Young Talents Fund of Henan Province (Grant No. 252300423500), and the Doctoral Startup Foundation of Henan Polytechnic University (Grant No. B2024-60).

\section*{Data Availability}
The numerical data and source code supporting this study are publicly available in Zenodo at \href{https://doi.org/10.5281/zenodo.21926498}{doi:10.5281/zenodo.21926498}.

\end{document}